\documentclass[a4paper,12pt]{amsart}
\usepackage{amsmath,amssymb,amsfonts,amsthm}
\usepackage{latexsym,mathrsfs}
\usepackage{graphicx}
\usepackage[all]{xy}
\input xypic
\usepackage{color}
\usepackage[pagebackref]{hyperref}
\hypersetup{colorlinks=true,linkcolor=red,citecolor=blue}
\usepackage[OT2,T1]{fontenc}
\usepackage[utf8]{inputenc}
\usepackage{lmodern}
\usepackage{microtype}
\usepackage{enumerate}
\usepackage{tikz-cd}
\usepackage{bm}
\usepackage{ulem}
\usepackage{mathtools}
\usepackage{tensor}
\theoremstyle{plain}

\newtheorem{corollary}[subsection]{{\bf Corollary}}
\newtheorem{proposition}[subsection]{{\bf Proposition}}

\theoremstyle{definition}

\theoremstyle{remark}

\numberwithin{equation}{subsection}

\DeclareMathOperator{\Sp}{Sp}

\DeclareBoldMathCommand{\bbot}{\bot}

\DeclareSymbolFont{cyrletters}{OT2}{wncyr}{m}{n}
\DeclareMathSymbol{\Sha}{\mathalpha}{cyrletters}{"58}

\newcommand{\Z}{\mathbb{Z}}

\begin{document}
\title[Tensor square of linear groups]{A note on the non-abelian tensor squares of linear groups}
\author{Primo\v z Moravec}
\address{{
Faculty of  Mathematics and Physics, University of Ljubljana,
and Institute of Mathematics, Physics and Mechanics,
Slovenia}}
\email{primoz.moravec@fmf.uni-lj.si}
\thanks{ORCID: \url{https://orcid.org/0000-0001-8594-0699}.}
\thanks{The author acknowledges the financial support from the Slovenian Research and Innovation Agency (ARIS), research core funding No. P1-0222, and project No. J1-50001.}
\date{\today}
\begin{abstract}
    We show that there exists a finitely generated linear group whose non-abelian tensor square is not linear. On the other hand, the non-abelian tensor square of the braid group $B_n$, $n>3$, is linear.
\end{abstract}
\maketitle


\section{Introduction}
\label{ss:intro}
\noindent
In this short note we observe that the non-abelian tensor square of a finitely generated linear group need not be a linear group. Then we show that the tensor square of the braid group $B_n$, $n>3$, is a linear group. These two results answer the Kourovka Notebook problem 19.9.

An example exhibiting the first statement, and some of the arguments leading to the second statement, were intially AI-generated, and then substantially simplified by the author in order to avoid involved homological arguments. As the reader will observe, both of the above results follow from standard facts. The note was written by the author, who takes the responsiblity for correctness.

Let $G$ be a group, acting on itself by conjugation from the left.  Its non-abelian tensor
square $G\otimes G$ is a group generated by symbols $g\otimes h$, $g,h\in G$, subject
to
\begin{align*}
 gg'\otimes h&=({}^{g}g'\otimes{}^{g}h)(g\otimes h),\\
 g\otimes hh'&=(g\otimes h)({}^{h}g\otimes{}^{h}h').
\end{align*}
This definition was introduced by Brown and Loday \cite{BrownLoday1987}.

The commutator map is
 $\kappa_G:G\otimes G\longrightarrow G'$ given by 
 $g\otimes h\longmapsto[g,h]$ is a homomorphism.
Let $\Delta(G)$ be the subgroup generated by all $g\otimes g$.  The
non-abelian exterior square of $G$ is
$G\wedge G=(G\otimes G)/\Delta(G)$. The commutator map $G\wedge G\to G'$ will be denoted by $\bar{\kappa}_G$.

 We will repeatedly use the following standard
structural facts:

\begin{proposition}[\cite{BrownLoday1987}]
    \label{prop:basictensor}
For every group $G$ there is a natural central extension
\begin{equation}\label{eq:exterior-extension}
 1\longrightarrow H_2(G;\Z)\longrightarrow G\wedge G
 \xrightarrow{\,\bar\kappa_G\,}G'\longrightarrow1.
\end{equation}
If $G$ is perfect, then \eqref{eq:exterior-extension} is the universal
central extension of $G$.
\end{proposition}

If a group $G$ is given by a free presentation $G=F/R$, then $G\wedge G$ can be naturally identified with $F'/[F,R]$, see, for example, \cite{BlythFumagalliMorigi2010}. Explicitly, an element $xR\wedge yR$ of $G\wedge G$ corresponds to the commutator $[x,y][F,R]$ in $F/[F,R]$.

\begin{proposition}[\cite{BlythFumagalliMorigi2010}]
    \label{prop:splitting}
Suppose that $G^{\rm ab}=G/G'$ is finitely generated and has no element of
order $2$.  Then
\[
 G\otimes G\cong \Delta(G)\times(G\wedge G),
 \qquad
 \Delta(G)\cong\Gamma(G^{\rm ab}),
\]
where $\Gamma$ is Whitehead's quadratic functor \cite{BrownLoday1987}.
\end{proposition}

For our purposes, it will suffice to know that $\Gamma(0)=0$ and $\Gamma(\Z)\cong\Z$.

\begin{corollary}
    \label{cor:fingen}
    If $G$ is a perfect finitely generated group, then $G\otimes G$ is finitely generated.
\end{corollary}

\section{Finitely generated linear groups}
\label{ss:lin}

\noindent
Let $g\ge 4$. The group $G=\Sp_{2g}(\Z)$ is a perfect finitely presented linear group. Hence, by Proposition \ref{prop:basictensor} and Proposition \ref{prop:splitting}, we have that $G\otimes G=G\wedge G=\widetilde{G}$, where $\widetilde{G}$ denotes the universal central extension of $G$. The group $\widetilde{G}$ is not residually finite. This essentially follows from a result of Deligne \cite{Deligne1978}, for a more direct statement see \cite{FunarPitsch2021}. As $G\otimes G$ is finitely generated, Mal'cev's theorem implies that it cannot be linear. QED.

\section{Braid groups}
\label{ss:braid}

\noindent
Let $n\ge 4$.
Let
\[
 B_n=\left\langle \sigma_1,\ldots,\sigma_{n-1}\ \middle|\
 \begin{array}{ll}
 \sigma_i\sigma_{i+1}\sigma_i=\sigma_{i+1}\sigma_i\sigma_{i+1},
     &1\leq i\leq n-2,\\
 \sigma_i\sigma_j=\sigma_j\sigma_i,&|i-j|\geq2
 \end{array}\right\rangle
\]
be the Artin braid group. The aim is to prove that the group $B_n\otimes B_n$ is linear.
At first note that $B_n^{\rm ab}\cong\Z$, hence Proposition \ref{prop:splitting} gives $B_n\otimes B_n\cong \Z\times (B_n\wedge B_n)$. Thus it suffices to prove that $B_n\wedge B_n$ is linear.




\begin{proposition}
    \label{prop:embed}
    The group $B_n\wedge B_n$ embeds into $B_n'\times (S_n\ \wedge S_n)$.
\end{proposition}


\begin{proof}
    The map
$p:B_n\longrightarrow S_n$ given by
$$\sigma_i\longmapsto s_i=\begin{pmatrix} i & i+1\end{pmatrix}$$
is an epimorphism. 
    Define
    $\iota :B_n\wedge B_n\to B_n'\times (S_n\wedge S_n)$
    by $$\iota=\bar{\kappa}_{B_n}\times (p\wedge p).$$
    Suppose $\iota(z)=1$ for some $z\in (B_n\wedge B_n)\setminus\{ 1\}$. Then $z\in\ker\bar{\kappa}_{B_n}$, and $z\in\ker p\wedge p$. Note that $\ker\bar{\kappa}_{B_n}\cong H_2(B_n,\Z)\cong C_2$ by a result of Huebschmann \cite[Theorem 5.1]{Huebschmann2012}. Furthermore, Huebschmann shows that $H_2(B_n,\Z)$ is generated by the abelian cycle $[\sigma_1\mid \sigma_3]-[\sigma_3\mid\sigma_1]$. One can infer from here that $\ker\bar{\kappa}_{B_n}=\langle \sigma_1\wedge \sigma_3\rangle$. This implies that $z=\sigma_1\wedge \sigma_3$, and therefore $s_1\wedge s_3=(p\wedge p)(z)=1$. On the other hand, $s_1\wedge s_3$ is a generator of $H_2(S_n,\Z)\cong C_2$. This essentially follows by inspecting the proof of Theorem 2.12.3 in \cite{Karpilovsky1987}, see also \cite[Proof of Lemma A.1]{Akita+2024}. This contradiction shows that $\iota$ is an embedding.
\end{proof}

To conclude the argument, note that $S_n\wedge S_n$, being finite, is obviously linear, whereas $B_n'$ is linear by a more profound result of Krammer \cite{Krammer2002}. Hence $B_n\wedge B_n$, and therefore also $B_n\otimes B_n$, is linear.

\end{document}